\documentclass[3p, authoryear, 10pt]{elsarticle}
\journal{Statistics \& Probability Letters}

\usepackage{amssymb}
\usepackage{amsmath}
\usepackage[colorlinks]{hyperref}
\usepackage{cleveref}

\newtheorem{lemma}{Lemma}
\newtheorem{theorem}{Theorem}
\newtheorem{corollary}{Corollary}
\newproof{proof}{Proof}

\newcommand{\E}{\mathbb E}
\newcommand{\Hh}{\mathcal H}
\newcommand{\one}{\mathbf{1}}
\newcommand{\R}{\mathbb R}
\newcommand{\X}{\mathcal X}

\begin{document}

\bibliographystyle{elsarticle-harv} 

\begin{frontmatter}
    \title{A simple geometric proof for\\ the characterisation of e-merging functions} 

    \author[1]{Eugenio Clerico}
    \ead{eugenio.clerico@gmail.com}

    \affiliation[1]{organization={Department of Statistics, University of Oxford},
    city={Oxford},
    country={UK}}

    \begin{abstract} 
        E-values offer a powerful framework for aggregating evidence across different (possibly dependent) statistical experiments. A fundamental question is to identify e-merging functions, namely mappings that merge several e-values into a single valid e-value. A simple and elegant characterisation of this function class was recently obtained by \cite{wang25only}, though via    technically involved arguments. This  note gives a short and intuitive geometric proof of the same characterisation, based on a supporting hyperplane argument applied to concave envelopes. We also show that the result holds even without imposing monotonicity in the definition of e-merging functions, which was needed for the existing proof. This shows that any non-monotone merging rule is automatically dominated by a monotone one, and hence extending the definition beyond the monotone case brings no additional generality.
    \end{abstract}
\end{frontmatter}

\section{Introduction}
Testing with \emph{e-values} is emerging as a flexible alternative to classical \emph{p-value}–based hypothesis testing \citep{ramdas2023game}. One of the main advantages of e-values is their ability to aggregate evidence across experiments more effectively than standard p-values, a feature that is particularly valuable in multiple and sequential testing (see, \textit{e.g.}, \citealt{vovk2021evalues,wang2022false,hartog2025family}). Formally, an e-value is the observed realisation of an e-variable, defined as a non-negative random variable whose expectation under the null hypothesis is upper bounded by one. Functions that take several e-values as inputs and combine them into a single ``summary'' e-value are called \emph{e-merging} functions. A key strength of this framework is that e-merging can be easily performed without imposing any assumptions on the dependence structure of the individual e-values.

In a recent work, \cite{wang25only} provided a complete characterisation of e-merging functions, by showing that every such function must be dominated by an affine map, and hence yielding a remarkably simple final description. Despite the simplicity of the result, the proof proposed by \cite{wang25only} is technically involved, relying on a minimax theorem and tools from optimal transport duality, and proceeding through several auxiliary results. In this short note, we show that the same characterisation can be obtained by a much more concise, and arguably more intuitive, geometric argument. Moreover, our approach yields a slightly stronger result: the same conclusion holds without explicitly assuming monotonicity in the definition of an e-merging function. In particular, any potentially non-monotone e-merging function is automatically dominated by a non-decreasing one. This refinement does not follow directly from the arguments in \cite{wang25only}, where the monotonicity assumption is needed in the proof.

\section{Characterisation of the e-merging functions}
Given a measurable space $\X$, and a collection $\Hh$ of probability measures on $\X$ (the so-called \emph{null hypothesis}), an \emph{e-variable} for $\Hh$ is a non-negative measurable function $E:\X\to[0,\infty)$, whose expectation is upper bounded by $1$ under every $P\in\Hh$. For $K\geq 1$ integer, following \cite{wang25only} we define an \emph{e-merging function} as a component-wise non-decreasing mapping
$F\,:\,[0,\infty)^K\to[0,\infty)$ such that, for any space $\X$, any null $\Hh$ on $\X$, and any $K$-tuple $\mathbf E = (E_1,\dots, E_K)$ of e-variables for $\Hh$, the composition
$$F\circ\mathbf E\;:\;x\mapsto F\big(E_1(x),\dots, E_K(x)\big)$$ is again an e-variable for $\Hh$. As a side remark, e-variables are typically allowed to take the value infinity, and one would therefore wish to consider mappings $[0,\infty]^K \to [0,\infty]$ in the definition of e-merging functions. However, as noted by \cite{vovk2021evalues}, despite the conceptual importance of infinite e-values, it is sufficient to restrict attention to finite e-values when discussing e-merging functions (see Appendix C therein).

\cite{wang25only} proved the following simple and direct characterisation of e-merging functions. A non-decreasing mapping $F:[0,\infty)^K\to [0,\infty)$ is an e-merging function if, and only if, there is a $K$-tuple $\mathbf w = (w_1,\dots, w_K)$ of non-negative coefficients that sum up to at most $1$, such that, for any $\mathbf u \in [0,\infty)^K$,
$$F(\mathbf u) \leq 1 + \sum_{k=1}^K w_k(u_k-1)\,.$$
The proof of \cite{wang25only} is based on a clever and non-trivial application of optimal transport duality. The key idea is to express the worst-case expectation of a merging function over all dependence structures as the solution of a dual optimization problem involving separable dominating functions. This reduction effectively decouples the coordinates. The problem is then reduced to a one-dimensional bound, which is handled using a sharp characterization of admissible one-dimensional e-merging functions. While the final result is simple, the proof requires several auxiliary intermediate results and a delicate minimax argument.

Here we present a much simpler proof for the characterisation of e-merging functions. The key idea is to consider the concave envelope of $F$ and use a classical hyperplane separation argument to show that it is dominated by a suitable affine function. Unlike the approach of \cite{wang25only}, this proof does not rely on the monotonicity assumption in the definition of an e-merging function.\footnote{A key technical step in the proof of \cite{wang25only} is that the analysis can be restricted to upper semi-continuous e-merging functions. This reduction relies on a lemma from \cite{vovk2021evalues}, which shows that every (monotone) e-merging function is dominated by an upper semi-continuous e-merging function. To prove this result, \cite{vovk2021evalues} crucially exploited the monotonicity assumption.} For this reason, we formulate the result in a slightly more general setting. Specifically, we call a \emph{generalised} e-merging function  any Borel measurable mapping $F:[0,\infty)^K\to[0,\infty)$ such that $F \circ \mathbf E$ is an $e$-variable for every $K$-tuple $\mathbf E$ of $e$-variables (with respect to any $\X$ and any $\Hh$).

\begin{theorem}\label{thm:main}
    Let $K\geq 1$ be an integer and $F:[0,\infty)^K\to[0,\infty)$ a generalised e-merging function. Then, there is a $K$-tuple $\mathbf w \in [0,1]^K$ such that $\sum_{k=1}^K w_k\leq 1$ and, for every $\mathbf u\in [0,\infty)^K$,
    $$F(\mathbf u) \leq 1 + \sum_{k=1}^K w_k(u_k-1)\,.$$
\end{theorem}

We remark that the monotonicity requirement in the definition of e-merging functions is nevertheless natural from a statistical perspective: larger e-values correspond to stronger evidence against the null, and one would therefore expect this ordering to be preserved under merging. However \Cref{thm:main} shows that this restriction is not essential at a mathematical level: any generalised e-merging function is always dominated by a (monotone) e-merging function. In this sense, allowing for non-monotonicity does not enlarge the effective class of admissible merging procedures, and there is therefore effectively no  gain in extending the definition beyond the monotone case.

\begin{corollary}
    Let $K\geq 1$ be an integer and $F:[0,\infty)^K\to[0,\infty)$ Borel-measurable. Then, $F$ is a generalised e-merging function if, and only if, there is a $K$-tuple $\mathbf w \in [0,1]^K$ such that $\sum_{k=1}^K w_k\leq 1$ and, for every $\mathbf u\in [0,\infty)^K$,
    $$F(\mathbf u) \leq 1 + \sum_{k=1}^K w_k(u_k-1)\,.$$
    In particular, every generalised e-merging function is dominated by a (monotone) e-merging function.
\end{corollary}
\begin{proof}
    Fix a $\mathbf w$ as in the statement, and let $F_{\mathbf{w}}:u\mapsto 1 + \sum_{k=1}^K w_k(u_k-1)$. It is straightforward to check that this is a (non-decreasing) e-merging function. Then, any Borel $F:[0,\infty)^K\to[0,\infty)$ dominated by $F_{\mathbf{w}}$ is clearly a generalised e-merging function. The other direction is precisely \Cref{thm:main}.
\end{proof}

\section{Geometric proof of the characterisation}
The idea to prove \Cref{thm:main} is rather simple. We consider the concave envelope $\hat F$ of the generalised e-merging function $F$ and we show that $\hat F(\one) \leq 1$, where $\one$ denotes the vector 
$(1,\dots,1)\in\R^K$. Then, the supporting hyperplane  theorem will immediately yield the desired conclusion.

First, let us establish a simple preliminary lemma, which will be the key ingredient of our proof. 
\begin{lemma}\label{lemma}
    Fix an integer $n\geq 1$, consider $n$ vectors $\mathbf{u}^{(1)},\dots, \mathbf{u}^{(n)}\in [0,\infty)^K$, and a weight vector $(q_1,\dots, q_n)\in[0,1]^n$, such that $\sum_{j=1}^n q_j =1$. We have the following implication: $$\sum_{j=1}^n q_j \mathbf{u}^{(j)}= \mathbf{1}\qquad\implies\qquad\sum_{j=1}^n q_j F(\mathbf{u}^{(j)})\leq 1\,.$$
\end{lemma}
\begin{proof}
    Let $\X = \{x_1,\dots, x_n\}$ be a  set made of $n$ distinct elements, and consider the hypothesis $\Hh = \{Q\}$ on $\X$, with $Q = \sum_{j=1}^n q_j \delta_{x_j}$. Define $K$ random variables $E_1,\dots, E_K$ on $\X$ via $E_k(x_j) = u^{(j)}_k$. Then, 
    $$\E_Q[E_k] = \sum_{j=1}^n q_j E_k(x_j) = \sum_{j=1}^n q_j u^{(j)}_k = 1\,,$$ and so each $E_k$ is an e-variable for $\Hh$. We have
    $$\sum_{j=1}^n q_j F\big(\mathbf{u}^{(j)}\big) = \sum_{j=1}^n q_j F\big(E_1(x_j),\dots, E_K(x_j)\big) = \E_Q[F(E_1,\dots, E_K)] \leq1\,,$$ since $F$ being a generalised e-merging function implies that $F(E_1,\dots, E_K)$ is an e-variable for $\Hh$.
\end{proof}

We now define a function $\hat F:[0,\infty)^K\to[0,\infty]$ by taking the supremum of $F$ over all finite convex combinations, namely
$$\hat F(\mathbf{u}) = \sup_{n\geq 1}\sup_{q\in \Delta_n}\sup\left\{\sum_{j=1}^n q_j F(\mathbf{u}^{(j)})\,\;:\,\;\mathbf{u}^{(j)}\in[0,\infty)^K\;\;\forall j\,,\;\;\sum_{j=1}^n q_j\mathbf{u}^{(j)} =  \mathbf{u}\right\}\,,$$
where $\Delta_n = \{q\in[0,1]^n\,:\,\sum_{j=1}^n q_j=1\}$. The case $n=1$ shows that $\hat F(\mathbf{u}) \geq F(\mathbf{u})$ for all $\mathbf{u}$.

By \Cref{lemma}, any candidate combination $\sum_{j=1}^n q_j \mathbf{u}^{(j)} = \mathbf{1}$ satisfies $\sum_{j=1}^n q_j F(\mathbf{u}^{(j)}) \leq 1$. Taking the supremum over all such combinations in the definition of $\hat F$ yields exactly
$$\hat F(\mathbf{1}) \leq 1\,.$$

We now show that $\hat F$ is concave (in fact, $\hat F$ is the \emph{concave envelope} of $F$). Fix any $\mathbf{v}, \mathbf{w} \in [0,\infty)^K$ and $\alpha \in [0,1]$. For any $\varepsilon > 0$, the definition of the supremum allows us to find integers $n, m \ge 1$, weights $p \in \Delta_n$, $q \in \Delta_m$, and sets  $\{\mathbf{v}^{(j)}\}_{j=1}^n$ and $\{\mathbf{w}^{(i)}\}_{i=1}^m$ such that $\sum_{j=1}^n p_j \mathbf{v}^{(j)} = \mathbf{v}$ and $\sum_{i=1}^m q_i \mathbf{w}^{(i)} = \mathbf{w}$, satisfying
$$\sum_{j=1}^n p_j F(\mathbf{v}^{(j)}) \geq \hat F(\mathbf{v}) - \varepsilon \qquad \text{and} \qquad \sum_{i=1}^m q_i F(\mathbf{w}^{(i)}) \geq \hat F(\mathbf{w}) - \varepsilon\,.$$
Now note that we can obtain an element in $r\in\Delta_{n+m}$ by reweighting and concatenating $q$ and $p$, that is $r=(\alpha p_1, \dots, \alpha p_n, (1-\alpha) q_1,\dots, (1-\alpha)q_m)\in\Delta_{n+m}$. We can then use $r$ to define a convex combination of all the vectors in $\{\mathbf{v}^{(j)}\}_{j=1}^n\cup\{\mathbf{w}^{(i)}\}_{i=1}^m$, and we get
$$\alpha \sum_{j=1}^n p_j F(\mathbf{v}^{(j)}) + (1-\alpha) \sum_{i=1}^m q_i F(\mathbf{w}^{(i)}) \geq \alpha \hat F(\mathbf{v}) + (1-\alpha) \hat F(\mathbf{w}) - \varepsilon\,.$$
Since the left side is a valid candidate in the supremum defining $\hat F(\alpha \mathbf{v} + (1-\alpha)\mathbf{w})$, we get that $$\hat F(\alpha \mathbf{v} + (1-\alpha)\mathbf{w}) \geq \alpha \hat F(\mathbf{v}) + (1-\alpha) \hat F(\mathbf{w}) - \varepsilon\,.$$ Letting $\varepsilon$ tend to $0$ we conclude that $\hat F$ is concave. 

Since $\hat F$ is concave, its hypograph $S=\{(\mathbf u,t)\in[0,\infty)^K\times\mathbb R:\ t\le \hat F(\mathbf u)\}$ is convex. Because $\hat F(\one)\leq 1$, the sets $S$ and $R=\{(\one,s): s>1\}$ are disjoint. By a classical separation theorem in finite-dimensional convex analysis (see, e.g., Chapter A, \S4.1 of \citealt{hiriart-urruty2001fundamentals}), these two disjoint convex sets can be separated by a hyperplane, meaning there exist a non-zero vector $(\mathbf b,a)\in\mathbb R^K\times\mathbb R$ and a scalar $c\in\mathbb R$ such that $\mathbf b\cdot\mathbf u + a t \le c \le \mathbf b\cdot\one + a s$, for all $(\mathbf u,t)\in S$ and all $s>1$.
Since the right-hand inequality holds for arbitrarily large $s$, we must have $a\ge 0$. If $a=0$, then
$\mathbf b\cdot\mathbf u\le c\le \mathbf b\cdot\one$ for all $ \mathbf u\in[0,\infty)^K$,
which is impossible unless $\mathbf b=\mathbf 0$, contradicting the fact that $(\mathbf b,a)\neq \mathbf 0$. Hence, it must be that $a>0$, which geometrically simply means that the hyperplane cannot be vertical.

Now fix $\mathbf u\in[0,\infty)^K$. Since the inequality $\mathbf b\cdot\mathbf u + a t \le c$ holds for every $t\le \hat F(\mathbf u)$, taking the supremum over such $t$ gives $\mathbf b\cdot\mathbf u + a\hat F(\mathbf u)\leq c$. On the other hand, the inequality $c\leq \mathbf b\cdot\one + a s$ holds for every $s>1$, so taking the infimum over $s>1$ yields $c\le \mathbf b\cdot\one + a$.
Dividing the two inequalities by $a>0$, we obtain
$$F(\mathbf u) \leq \hat F(\mathbf u)\leq 1+\sum_{k=1}^K w_k(u_k-1)\,,$$
where $\mathbf{w}=-\mathbf{b}/a$. Note that $\hat F$ is non-negative on the whole domain $[0,\infty)^K$, and so the components of $\mathbf{w}$ are non-negative and have sum at most $1$, as required. 

\paragraph{Acknowledgements.} The author acknowledges the use of ChatGPT (OpenAI) and Gemini (Google) for assistance in exploring proof ideas and for improving the clarity of the exposition. All results and proofs were independently verified by the author, who takes full responsibility for the content.
\bibliography{bib.bib}

@misc{hartog2025family,
      title={Family-wise Error Rate Control with E-values}, 
      author={Will Hartog and Lihua Lei},
      year={2025},
      eprint={2501.09015},
}

@article{ramdas2023game,
        author = {Aaditya Ramdas and Peter Gr{\"u}nwald and Vladimir Vovk and Glenn Shafer},
        title = {Game-Theoretic Statistics and Safe Anytime-Valid Inference},
        volume = {38},
        journal = {Statistical Science},
        number = {4},
        publisher = {Institute of Mathematical Statistics},
        pages = {576 -- 601},
        year = {2023},
}

@book{hiriart-urruty2001fundamentals,
  title     = {Fundamentals of Convex Analysis},
  author    = {Hiriart-Urruty, Jean-Baptiste and Lemar{\'e}chal, Claude},
  year      = {2001},
  publisher = {Springer},
  address   = {Berlin},
}

@article{vovk2021evalues,
  author = {Vovk, Vladimir and Wang, Ruodu},
  title = {E-values: Calibration, combination, and applications},
  journal = {The Annals of Statistics},
  volume = {49},
  issue = {3},
  pages = {1736-1754},
  year = {2021},
  doi = {10.1214/20-AOS2020}
}

@article{wang2022false,
  title={False discovery rate control with e-values},
  author={Wang, Ruodu and Ramdas, Aaditya},
  journal={Journal of the Royal Statistical Society Series B: Statistical Methodology},
  volume={84},
  number={3},
  pages={822-852},
  year={2022},
  publisher={Oxford University Press},
  doi={10.1111/rssb.12499}
}

@article{wang25only,
  author  = {Wang, Ruodu},
  title   = {The Only Admissible Way of Merging Arbitrary e-Values},
  journal = {Biometrika},
  volume  = {112},
  number  = {2},
  year    = {2025},
  pages   = {asaf020},
}
\end{document}